\documentclass[12pt]{amsart}
\usepackage{amsmath,amssymb,amsfonts,amsthm,amstext,verbatim,url}
\usepackage{graphicx}
\RequirePackage[colorlinks,citecolor=blue,urlcolor=blue]{hyperref}

\theoremstyle{plain}

\newtheorem{theorem}{Theorem}

\numberwithin{equation}{section}
\numberwithin{theorem}{section}
\numberwithin{figure}{section}

\newcounter{mycount}

\newenvironment{letlist}{\begin{list}{\rm(\alph{mycount})}%
   {\usecounter{mycount}\labelwidth=1cm\itemsep 0pt}}{\end{list}}

\newcommand\s{\sigma}

\newcommand\oo{\infty}

\newcommand\LL{{\mathbb L}}
\newcommand\HH{{\mathbb H}}
\newcommand\NN{{\mathbb N}}
\newcommand\sL{{\mathcal L}}
\newcommand\sA{{\mathcal A}}

\newcommand\sAc{\sA_{\text{\rm c}}}

\newcommand\sE{{\mathcal E}}

\newcommand\ZZ{{\mathbb Z}}
\newcommand\RR{{\mathbb R}}
\newcommand\Ze{Z_{\text{\rm e}}}
\newcommand\sEe{\sE_{\text{\rm e}}}

\newcommand\wt{\widetilde}

\newcommand\sm{\setminus}
\renewcommand\a{\alpha}

\newcommand\Si{\Sigma}

\newcommand\g{\gamma}
\newcommand\resp{respectively}
\newcommand\fish{F}
\newcommand\bb{\mathrm{b}}
\newcommand\ww{\mathrm{w}}

\begin{document}
\title[SAWs and the Fisher transformation]
{Self-avoiding walks and the\\ Fisher transformation}
\author[Grimmett]{Geoffrey R.\ Grimmett}
\address{Statistical Laboratory, Centre for
Mathematical Sciences, Cambridge University, Wilberforce Road,
Cambridge CB3 0WB, UK} 
\email{\{g.r.grimmett, z.li\}@statslab.cam.ac.uk}
\urladdr{http://www.statslab.cam.ac.uk/$\sim$grg/}
\urladdr{http://www.statslab.cam.ac.uk/$\sim$zl296/}

\author[Li]{Zhongyang Li}

\begin{abstract}
The Fisher transformation acts on cubic graphs by replacing each vertex by a triangle.
We explore the action of the Fisher transformation
on the set of self-avoiding walks of a cubic graph.
Iteration of the transformation yields a 
sequence of graphs with common critical exponents, and with
connective constants converging geometrically to the golden mean.

We consider the application of the Fisher transformation
to one of the two classes of vertices of a bipartite cubic graph. The connective constant
of the ensuing graph may be expressed in terms of that of the initial graph. When
applied to the hexagonal lattice, this identifies a further lattice whose
connective constant may be computed rigorously.
\end{abstract}

\date{24 August 2012} 


\maketitle

\section{Introduction}\label{sec:intro}
A \emph{self-avoiding walk} (abbreviated to SAW) on a graph $G$
is a path that visits no point more than once. SAWs
were introduced in the chemical theory of polymerization (see Flory \cite{f}),
and their critical behaviour has been studied since by mathematicians and physicists
(see, for example, the book  \cite{ms} of Madras and Slade). The exponential
rate of growth of the number of SAWs is given by the so-called \emph{connective constant}
$\mu=\mu(G)$ of the graph.
Only few graphs of interest have
connective constants that are known exactly.

We explore the action of the Fisher transformation on
the set of SAWs of a cubic graph $G$. The transformation maps $G$ to
a new graph $F(G)$.
We have two sets of results.
First, the connective constants of $G$ and $F(G)$ 
satisfy a simple functional relation, and in addition, 
three of the principal critical exponents 
are invariant under the transformation. In addition, under repeated applications of 
the Fisher transformation, the graphs converge to a version
of the Sierpinski gasket, and the connective constants converge 
geometrically to the golden mean. 
See Theorems \ref{thm:main2} and \ref{thm:main2'} for formal
statements of these results.

\begin{figure}[htb]
 \centering
\includegraphics[width=0.8\textwidth]{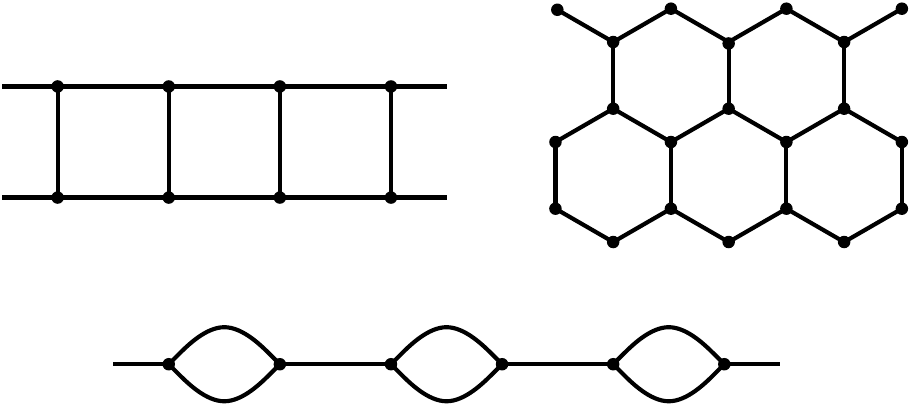}
  \caption{Three cubic graphs: the  (doubly-infinite) ladder graph $\LL$;
the hexagonal lattice $\HH$; the \emph{loop graph} $\sL_3$ obtained from $\ZZ$ by joining
every alternating pair of consecutive vertices by $2$ parallel edges.}
  \label{fig:ladder-hex}
\end{figure}

Our second set of results concerns the application
of the Fisher transformation to a bipartite graph $G$ one of whose vertex-sets
is cubic. As before, the ensuing connective constant may be expressed in terms
of that of $G$, and the critical exponents are invariant.
When applied to the hexagonal lattice $\HH$
(see Figure \ref{fig:ladder-hex}), this yields the lattice $\wt \HH$
illustrated in Figure \ref{fig:newlatt}. 
Nienhuis's proposed value $\mu(\HH) = \sqrt{2+\sqrt 2}$
has been proved recently by Duminil-Copin and Smirnov \cite{ds},
and the  value of $\mu(\wt\HH)$ may be deduced rigorously
from this, namely as
the root of the equation 
$$
x^{-3}+x^{-4} = \frac1{2+\sqrt 2}.
$$ 
See Theorem \ref{thm:main3}.

Section \ref{sec:notation} is devoted to basic definitions.
The Fisher transformation, and
its action on counts of SAWs, is described
in Section \ref{sec:fisher}, and our Theorems \ref{thm:main2}--\ref{thm:main3}
are stated there. 
The proofs of results are found
in Sections \ref{sec:main2} and \ref{sec:proof3}.

\begin{figure}[htb]
 \centering
    \includegraphics[width=0.5\textwidth]{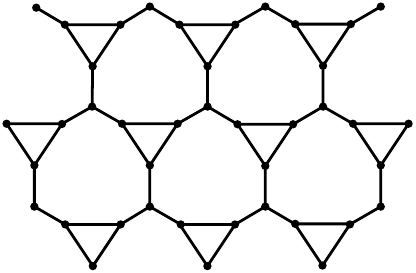}
  \caption{The lattice $\wt\HH$ derived from the hexagonal lattice $\HH$ by
applying the Fisher transformation at alternate vertices. Its connective constant
$\wt\mu$ is the root of the equation
$x^{-3}+x^{-4} = 1/(2+\sqrt 2)$.}
  \label{fig:newlatt}
\end{figure}

In a companion paper \cite{GrL1}, we study inequalities for the connective constants of
regular graphs. For an infinite, connected,  cubic, quasi-transitive graph $G$
(possibly with parallel edges), it is elementary that
\begin{equation}\label{ineq}
1 \le \mu(G) \le 2.
\end{equation}
If such $G$ is vertex-transitive, it is proved in \cite{GrL1} that 
$\sqrt 2 \le \mu(G) \le 2$, with equalities for the loop graph
$\sL_3$ of Figure \ref{fig:ladder-hex} and the $3$-regular 
tree, \resp. 

\section{Notation}\label{sec:notation}

All graphs studied henceforth in this paper will be assumed infinite, connected,
and simple (in that they have neither loops nor multiple edges). 
An edge $e$ with endpoints $u$, $v$ is written $e=\langle u,v \rangle$.
If $\langle u,v \rangle \in E$, we call $u$ and $v$ \emph{adjacent}
and write $u \sim v$.
The \emph{degree} of vertex $v$ is the number of edges
incident to $v$, denoted $\deg(v)$. A graph is called \emph{cubic}
if all vertices have degree $3$.
The \emph{graph-distance} between two vertices $u$, $v$ is the number of edges
in the shortest path from $u$ to $v$, denoted $d_G(u,v)$.

The automorphism group of the graph $G=(V,E)$ is
denoted $\sA=\sA(G)$.
The graph $G$ is
called \emph{quasi-transitive} if there exists a finite subset $W \subseteq V$ such that,
for $v \in V$ there exists $\a\in\sA $ such
that $\a v \in W$. We call such $W$ a \emph{fundamental domain},
and shall normally (but not invariably) take $W$
to be minimal with this property. The graph is 
called \emph{vertex-transitive} 
(or \emph{transitive}) if the singleton set $\{v\}$ is a fundamental
domain for some (and hence all) $v \in V$.

A \emph{walk} $w$ on $G$ is
an alternating sequence $v_0e_0v_1e_1\cdots e_{n-1} v_n$ of vertices $v_i$
and edges $e_i$ such that $e_i=\langle v_i, v_{i+1}\rangle$.
We write $|w|=n$ for the \emph{length} of $w$, that is, the number of edges in $w$.

Let $n \in \NN$, the natural numbers. 
A $n$-step \emph{self-avoiding walk} (SAW) 
on $G$ is  a walk containing $n$ edges
that includes no vertex more than once.
Let $\s_n(v)$ be the number of $n$-step SAWs 
 starting at $v\in V$. It was shown by Hammersley \cite{jmhII}
that, if $G$ is quasi-transitive, there exists a constant $\mu=\mu(G)$,
called the \emph{connective constant} of $G$,
such that
\begin{equation}
\label{connconst}
\mu= \lim_{n\to\oo}  \s_n(v)^{1/n}, \qquad v \in V.
\end{equation}

It will be convenient to consider also SAWs starting at `mid-edges'. We identify the
edge $e$ with a point (also denoted $e$) placed at the middle of $e$, 
and then consider walks that start and end at 
these mid-edges. Such a  walk is \emph{self-avoiding} if it visits no vertex or mid-edge 
more than once, and its \emph{length} is the number of vertices visited.

The minimum of two reals $x$, $y$ is denoted $x \wedge y$, and the maximum
$x \vee y$.

\section{Fisher transformation} \label{sec:fisher}

Let $G=(V,E)$ be a simple graph and let $v \in V$ have degree $3$. 
The so-called \emph{Fisher transformation} acts on $v$ by 
replacing it by a triangle, as illustrated in 
Figure \ref{fig:fisher}. This transformation has been valuable in the study of
the relations between Ising, dimer, and general vertex models 
(see \cite{bdet,fisher,zli,li}), and more recently of SAWs
on the Archimedean lattice denoted $(3,12^2)$ (see \cite{g,jg}).  
In the remainder of this paper, we make use of the Fisher transformation in the context of 
SAWs and the connective constant. It will be applied to cubic graphs, of which the 
hexagonal and square/octagon lattices are examples.

\begin{figure}[htb]
 \centering
    \includegraphics[width=0.5\textwidth]{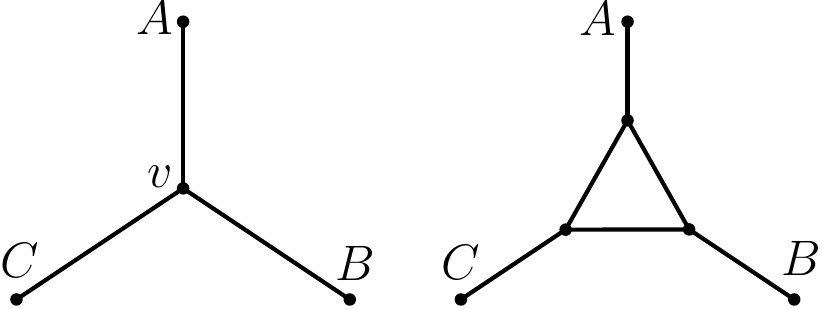}
  \caption{The Fisher triangulation of the star. Any triangle thus created
is called a \emph{Fisher triangle}.}
  \label{fig:fisher}
\end{figure}

It is convenient to work with graphs with well-defined connective constants, and
to this end we assume 
that $G=(V,E)$ is quasi-transitive and connected, so that its connective constant
is given by \eqref{connconst}.
We write $\fish(G)$ for the graph obtained
from the cubic graph $G$ by applying the Fisher transformation at every vertex. 
Obviously the automorphism group of $G$ induces an automorphism subgroup
of $\fish(G)$. 
We write $\phi= \tfrac12(\sqrt{5}+1)$ for the golden mean.
The next theorem may be known to others.

\begin{theorem}\label{thm:main2}
Let $G$ be an infinite, quasi-transitive, connected, cubic graph, and consider the
sequence $(G_k: k=0,1,2,\dots)$ defined by $G_0=G$ and $G_{k+1} = \fish(G_k)$.
Then
\begin{letlist}
\item The connective constants $\mu_k$ of the $G_k$ satisfy
$\mu_k^{-1} = g(\mu_{k+1}^{-1})$ where
$g(x)= x^2 + x^3$.
\item 
The sequence $\mu_k$ converges monotonely to $\phi$,
and 
\begin{equation*}
- \left(\tfrac 47\right)^k \le \mu_k^{-1} - \phi^{-1} 
\le \bigl[\tfrac12(7-\sqrt 5)\bigr]^k, \qquad k \ge 1.
\end{equation*}
\end{letlist}
\end{theorem}

Theorem \ref{thm:main2} provokes the question of the existence of a 
graph limit of repeated application of the Fisher transformation. 
It is easily seen that 
the limiting graph comprises two copies of the Sierpinski gasket,
as illustrated in 
Figure \ref{fig:fisherlimit}.

\begin{figure}[htb]
 \centering
    \includegraphics[width=0.7\textwidth]{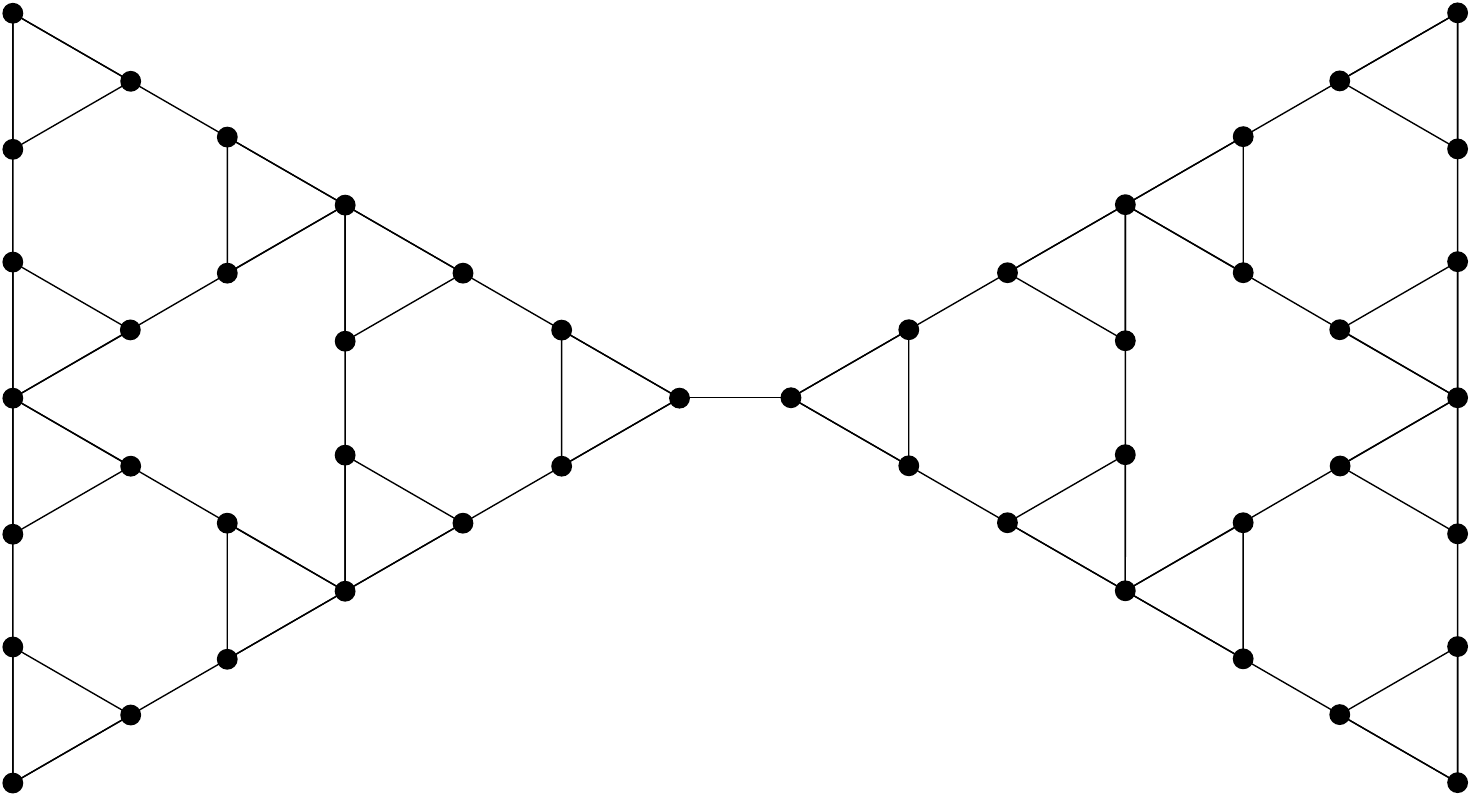}
  \caption{Through repeated application of the 
Fisher transformations to a single
edge with endvertices of degree three, one arrives
at a graph comprising two Sierpinski gaskets.}
  \label{fig:fisherlimit}
\end{figure}

By Theorem \ref{thm:main2}(b), either $\mu_k \downarrow \phi$ or $\mu_k \uparrow \phi$.
The \emph{decreasing} limit holds  if and only if $\mu_0 \ge\phi$. We present no satisfactory
characterization of graphs $G$ for which $\mu(G) \ge \phi$ beyond noting
that this holds whenever $G$ contains as a subgraph a copy of a graph with connective constant $\phi$,
such as the ladder graph $\LL$ (or the semi-infinite ladder graph) of Figure \ref{fig:ladder-hex}.
Furthermore, if $\mu(G) > \phi$ and $\wt G$ is obtained from $G$ by a sequence of
Fisher transformations, then $\mu(\wt G) \ge \phi$.

We turn to the topic of \emph{critical exponents}, beginning with
a general introduction for the case when there exists a periodic, locally finite embedding
of $G$ into $\RR^d$ with $d \ge 2$. The case of general $G$ has not not been 
studied extensively, and most attention has been paid to the 
hypercubic lattice $\ZZ^d$.
It is believed 
 (when $d \ne 4$) that there is a power-order correction, 
in the sense that there exists $A_v>0$ and an exponent 
$\g \in \RR$ such that
\begin{equation}\label{defgamma}
\s_n(v)\sim A_v n^{\gamma-1}\mu^n \qquad \text{as } n \to\oo,
\qquad v \in V.
\end{equation}
Furthermore, 
the value of the exponent $\g$ is believed to depend on $d$ and not further on the 
choice of graph $G$.
When $d=4$, \eqref{defgamma} should hold with $\gamma=1$
and subject to the inclusion
on the right side of the logarithmic correction
term  $(\log n)^{1/4}$.
See \cite{bdgs,ms} for  accounts of critical exponents for SAWs.

Let $v \in V$ and
\begin{equation}
Z_{v,w}(x)=\sum_{n=0}^\oo \s_n(v,w)x^k, \qquad w \in V,\ x>0,
\label{defz}
\end{equation} 
where $\s_n(v,w)$ is the number of $n$-step SAWs with endpoints
$v$, $w$.
It is known under certain circumstances that the generating 
functions $Z_{v,w}$ have radius of convergence $\mu^{-1}$ 
(see \cite[Cor.\ 3.2.6]{ms}),
and it is believed that there exists an exponent $\eta$ and constants 
$A'_v>0$  such that
\begin{equation}\label{defeta}
Z_{v,w}(\mu^{-1})\sim A_v'{d_G(v,w)}^{-(d-2+\eta)}
\qquad \text{as } d_G(v,w)\to\oo.
\end{equation}

Let $\Si_n(v)$ be the set of $n$-step SAWs from $v$,
and write $\langle \cdot \rangle_n^v$ for expectation with respect to
uniform measure on $\Sigma_n(v)$. Let
$\|\pi\|$ be the graph-distance between the endpoints of 
a SAW $\pi$.
It is believed (when $d\ne 4$) that
there exists an exponent $\nu$ and constants
$A''_v>0$ such that  
\begin{equation}\label{defnu}
\langle \|\pi\|^2\rangle_n^v \sim A_v'' n^{2\nu}, \qquad v \in V.
\end{equation}
As above, this should hold for $d=4$ with $\nu=\frac12$ and subject to the
inclusion of the correction term $(\log n)^{1/4}$.

The above three exponents  are believed to be related to one another 
through the so-called \emph{Fisher relation}
\begin{equation}\label{fisher-rel}
\gamma = \nu(2-\eta).
\end{equation}

It is convenient to work with 
definitions of critical exponents that do not depend on
an assumption of dimensionality, and thus we proceed as follows.
Let $G$ be an infinite, connected, quasi-transitive graph with 
connective constant $\mu$ and fundamental domain $W$.
Let $X$ be the set of edges incident to vertices in  $W$,
and let $\Sigma$ be the set of SAWs on $G$ starting at mid-edges in $X$.
We define the function
\begin{equation*}\label{expatn}
Y(x,y)=
\sum_{\pi\in\Sigma}
\frac{x^{|\pi|}}{|\pi|^y}, \qquad x> 0,\ y\in\RR.
\end{equation*}
(The denominator is interpreted as 1 when $|\pi|=0$.)
For fixed $x$, $Y(x,y)$ is non-increasing in $y$. 
Let $\g=\g(G)\in[-\oo,\oo]$ be such that
$$
Y(\mu^{-1},y) \begin{cases} =\infty &\text{if } y<\gamma,\\ 
<\oo &\text{if } y>\g.
\end{cases}
$$
We shall assume that $-\oo < \g<\oo$. 
It will be convenient at times to  assume more about
the number $\s_n$ of $n$-step SAWs from mid-edges in $X$, 
namely that   there exist constants $C_i=C_i(W) \in(0,\infty)$ and a slowly varying function $L$ 
such that
\begin{equation}\label{def:slow}
C_1 L(n)n^{\gamma-1}\mu^n\leq \sigma_{n}\leq C_2 L(n)n^{\gamma-1}\mu^n, \qquad n \ge 1.
\end{equation}

Let
\begin{equation}\label{def0}
V(z)=\sum_{n=1}^{\infty}\frac{1}{n^{2z+1}} \langle \|\pi\|^2 \rangle_n,
\qquad z \in [-\oo,\oo],
\end{equation}
where $\langle \cdot \rangle_n$ denotes the uniform average over the
set $\Sigma_n$ of $n$-step SAWs in $\Sigma$.
Thus, $V(z)$ is non-increasing in $z$, and we let  
$\nu=\nu(G)\in[-\oo,\oo]$ be such that
$$
V(z) \begin{cases} =\infty &\text{if } z < \nu,\\
<\infty &\text{if } z > \nu.
\end{cases}
$$

Let $\a W$ denote the image of $W$ under an automorphism $\a\in\sA$, with incident
edges $\a X$, and let
\begin{equation*}
Z_\a(x)=\sum_{\pi\in\Sigma(\a)} x^{|\pi|},
\end{equation*}
where $\Sigma(\a)$ is the subset of $\Sigma$ containing SAWs ending 
at mid-edges in $\a X$.
We assume there exists $\eta =\eta(G) \in [-\oo,\oo]$ such that,
for any sequence of automorphisms $\a$ satisfying $d_G(W,\a W)\to\oo$,
\begin{equation}\label{dd1}
Z_\a(\mu^{-1})d_G(W,\a W)^{w} 
\begin{cases} \to 0 &\text{if } w<\eta,\\
\to\infty &\text{if } w>\eta. 
\end{cases}
\end{equation}
The $\eta$ of \eqref{defeta} should agree with that defined here when $d=2$.

It is easily seen that the values of $\g$, $\eta$, $\nu$ do not depend on the choice of
fundamental domain $W$. 

We consider now the effect on critical exponents of
the Fisher transformation. 
Let $W_0$ be a minimal fundamental domain of $G_0 := G$, 
with incident edge-set $X_0 := X$ as above. Write $W_1=F(W_0)$, 
the set of vertices of the 
triangles formed by the Fisher transformation at vertices in $W_0$,
and $X_1$ for the set of edges of $G_1$ incident to vertices in $W_1$.
It may be seen that $W_1$ is a fundamental domain of $G_1$.

\begin{theorem}\label{thm:main2'}
Let $G_0$ be an infinite, quasi-transitive, connected, cubic graph.
Assume that $|\g(G_0)|<\oo$ and that $\eta(G_1)$ exists.
\begin{letlist}
\item The exponents $\g$, $\eta$ of $G_0$ and $G_1$ are equal.
\item  Let $\sigma_{n,k}$ be
the number of $n$-step SAWs on $G_k$ from mid-edges in $X_k$.
Assume the $\s_{n,k}$ satisfy \eqref{def:slow} for constants
$C_{i,k}$ and a common slowly varying function $L$. Then
the exponents $\nu$ of $G_0$ and $G_1$ are equal.
\end{letlist}
\end{theorem}

Our final result concerns the effect of the Fisher transformation 
when applied to just one of the vertex-sets of a bipartite graph.
Let $G=(V,E)$ be bipartite with vertex-sets $V_1$, $V_2$ coloured white and black, 
\resp. We think of $G$ as a graph together with a colouring $\chi$, and 
the \emph{coloured-automorphism group} $\sAc=\sAc(G)$ of the pair $(G,\chi)$ is the set
of maps $\phi:V \to V$ which preserve both graph structure and colouring.
The coloured graph is \emph{quasi-transitive} if there exists a finite 
subset $W \subseteq V$ such
that: for all $v\in V$, there exists $\a\in\sAc$ such that $\a v\in W$ and $\chi(v) = \chi(\a v)$. 
As before, such a set $W$ is called a \emph{fundamental domain}.

\begin{theorem}\label{thm:main3}
Let $G$ be an infinite,  connected,
bipartite graph with vertex-sets coloured black and white, and suppose that
the coloured graph $G$ is quasi-transitive, and every black vertex has degree $3$. 
Let $\wt G$ be obtained by applying the Fisher transformation at each black vertex. 
\begin{letlist}
\item The connective constants $\mu$ and $\wt\mu$ of $G$ and $\wt G$, \resp, satisfy 
$\mu^{-2} = h(\wt\mu^{-1})$ where
$h(x) = x^3 + x^4$.
\item 
Under the corresponding assumptions of Theorem \ref{thm:main2'}, the exponents $\g$, $\eta$, $\nu$ are the same for $G$ 
as for $\wt G$.
\end{letlist}
\end{theorem}
 
Theorem \ref{thm:main3}(a) implies an exact value of a connective constant that does
not appear to have been noted previously. Take $G=\HH$, the hexagonal
lattice with connective constant $\mu= \sqrt{2+\sqrt 2} \approx 1.84776$,
see \cite{ds}. The decorated
lattice $\wt \HH$ is illustrated in Figure \ref{fig:newlatt}, and has connective constant $\wt\mu$
satisfying $\mu^{-2} = h(\wt\mu^{-1})$, 
which may be solved to obtain $\wt\mu \approx 1.75056$.

The proofs of Theorems \ref{thm:main2}--\ref{thm:main2'} and 
\ref{thm:main3} are found  in
Sections \ref{sec:main2} and \ref{sec:proof3}, \resp.

\section{Proof of Theorems \ref{thm:main2}--\ref{thm:main2'}}\label{sec:main2}

\begin{proof}[Proof of Theorem \ref{thm:main2}]

Let $G_0=(V_0,E_0)$ be an infinite, connected, quasi-transitive, cubic graph. 
The graph $G_1 = F(G_0)$ is also quasi-transitive
 and cubic. It suffices for part (a) to show that the connective constants $\mu_k$
of the $G_k$ satisfy 
\begin{equation}\label{10}
g(\mu_1^{-1}) = \mu_0^{-1}.
\end{equation}
By \eqref{ineq}, $\mu_k \in [1,2]$ for $k=1,2$.

Let $W_0$ be a minimal fundamental domain of $G_0$, and let 
$X_0$ be the subset of $E_0$ comprising all edges incident 
to vertices in $W_0$. Write $W_1=F(W_0)$, the set of vertices of the 
triangles formed by the Fisher transformation at vertices in $W_0$,
and $X_1$ for the set of edges of $G_1$ incident to vertices in $W_1$.
It may be seen that $W_1$ is a fundamental domain of $G_1$.

It is convenient to work with SAWs that start and end at mid-edges.
Note that the mid-edges 
of $E_0$ (\resp, $X_0$) may be viewed as  mid-edges of 
$E_1$ (\resp, $X_1$). 

For $k=0,1$, the partition functions of SAWs on $G_k$ are the polynomials
\begin{equation*}
Z_k(x)=\sum_{\pi\in \Sigma_k} x^{|\pi|}, \qquad x >0,
\end{equation*}
where the sum is over the set $\Sigma_k$ of
SAWs starting at mid-edges of $X_k$. Similarly, we set
\begin{equation*}
Z_1^*(x) = \sum_{\pi \in \Sigma_1^*} x^{|\pi|},
\end{equation*}
where the sum is over the set $\Sigma_1^*$ of SAWs on 
$G_1$ starting at mid-edges of $X_0$ and ending at mid-edges of $E_0$.
For $k=0,1$,
\begin{equation}\label{225}
Z_k(x) 
\begin{cases} <\oo &\text{if } x < \mu_k^{-1},\\
=\oo &\text{if } x > \mu_k^{-1}.
\end{cases}
\end{equation}

The following basic argument
formalizes a method known already in the special case of the 
hexagonal lattice, see for example \cite{g,jg}.
Since $\Sigma_1^* \subseteq \Sigma_1$, we have 
\begin{equation}\label{221}
Z_1^*(x) \le Z_1(x).
\end{equation}
Let  $N(\pi)$ be the number of endpoints of a SAW 
$\pi\in \Sigma_1$ 
that are mid-edges of $E_0$. 
The set $\Sigma_1$ may be partitioned into three sets.
\begin{letlist}
\item If $N(\pi)=2$, then $\pi$ contributes
to $Z_1^*$.
\item $\pi$ may be a  walk within a single Fisher triangle.
\item If (b) does not hold and $N(\pi)\le 1$, any endpoint not in $E_0$ may
be moved by one, two, or three steps along $\pi$ to obtain a shorter SAW in $\Sigma_1^*$.
\end{letlist}
By considering the numbers of SAWs in each subcase of (c), we find that
\begin{equation}\label{222}
Z_1(x) \leq  
[1+2x+2x^2+2x^3]^2 Z_1^*(x) +6|W_0|(1+x+x^2).
\end{equation}
where the last term corresponds to case (b).
By \eqref{221}--\eqref{222}, 
$$
Z_1(x)<\oo \quad\Leftrightarrow\quad Z_1^*(x)<\oo,
$$
so that, by \eqref{225}, 
\begin{equation}\label{223}
Z_1^*(x) \begin{cases} <\oo &\text{if } x < \mu_1^{-1},\\
=\oo &\text{if } x > \mu_1^{-1}.
\end{cases}
\end{equation}

With a SAW in $\Sigma_1^*$ we associate a SAW in $\Sigma_0$ 
by shrinking each Fisher triangle to a vertex. Each $n$-step SAW  
in $\Sigma_0$ arises thus from $2^n$ SAWs in $\Sigma_1^*$, because 
each triangle may be circumnavigated in either of $2$ directions. 
Therefore,
\begin{equation}\label{r2}
Z_0(x^2(1+x))=Z_{1}^*(x),
\end{equation}
and \eqref{10} follows by \eqref{225} and \eqref{223}.

We turn to Theorem \ref{thm:main2}(b). By \eqref{ineq},
$\mu_0^{-1}  \in [\frac12, 1]$. The function
$g$ is a bijection from $[\frac12,1]$ to $[\frac38,2]$.
Furthermore, $g$ is strictly convex on $[\frac12,1]$ with
fixed point $\phi^{-1}$. By \eqref{10} applied iteratively,
$\mu_k^{-1} \to \phi^{-1}$ as $k\to\oo$, and the limit is monotone. 
The bounds on $\mu_k^{-1} - \phi^{-1}$ follow
from the facts that $g'(\frac12) = \frac74$ and 
$g'(\phi^{-1}) = \frac12(7-\sqrt 5)$.
\end{proof}

\begin{proof}[Proof of Theorem \ref{thm:main2'}]

Let
\begin{equation*}\label{expatn2}
Y_{k}(x,y)=
\sum_{\pi\in \Sigma_k}
\frac{x^{|\pi|}}{|\pi|^y},
\quad Y_k^*(x,y)=\sum_{\pi\in\Sigma^*_k} \frac{x^{|\pi|}}{|\pi|^y}, 
 \qquad x>0,\ y\in\RR,
\end{equation*}
where the denominator is interpreted as $1$ when $|\pi|=0$.
Since $\Sigma_1^* \subseteq \Sigma_1$,
\begin{equation*}
Y_1^*(x,y)\leq Y_1(x,y).
\end{equation*}
Since every SAW in $\Sigma_1 \sm \Sigma_1^*$ either is an extension of a 
SAW in $\Sigma_1^*$ at the starting point, or endpoint (or both), 
by at most $3$ steps, or is a short
walk in a single Fisher triangle,
\begin{equation*}
Y_1(x,y)\leq 7^{|y|}[1+2x+2x^2+2x^3]^2 Y_1^*(x,y)+6|W_0|\left(1+x+\frac{x^2}{2^y}\right).
\end{equation*}

Therefore,
\begin{equation}\label{gg1}
Y_1^*(x,y)<\oo \quad \Leftrightarrow \quad  Y_1(x,y)<\oo.
\end{equation}

As in the previous proof, any $n$-step SAW in $\Sigma_0$ 
gives rise to $2^n$ SAWs in $\Sigma_1^*$,
and conversely any SAW in $\Sigma^*_1$ gives rise to a SAW in $\Sigma_0$ 
by shrinking each triangle to a vertex. For $n \ge 1$, the contribution of an $n$-step SAW
$\pi\in \Sigma_0$ to $Y_0(x,y)$ is ${x^n}/{n^y}$, and to $Y_1^*(x,y)$ is 
$$
T_n := \sum_{l=0}^n \binom nl \frac{x^{2n+l}}{(2n+l)^y}.
$$
Since
\begin{equation*}
C \frac{[x^2(1+x)]^n}{n^y} 
\leq T_n
\leq D \frac{[x^2(1+x)]^n}{n^y},
\end{equation*}
where $C= 2^{-y} \wedge 3^{-y}$  and $D= 2^{-y}\vee 3^{-y}$,
we have that
\begin{equation*}
C \wt Y_0(x^2(1+x),y)\leq \wt Y_1^*(x,y)\leq D  \wt Y_0(x^2(1+x),y),
\end{equation*}
where $\wt S$ denotes the summation $S$ without the $n=0$ term.
Therefore, 
$$
Y_1^*(x,y)<\oo \quad \Leftrightarrow \quad Y_0(x^2(1+x),y)<\oo.
$$
By \eqref{gg1} and Theorem \ref{thm:main2}(a), 
$\g(G_1) = \g(G_0)$.

Let $\|\pi\|_k$ be the graph-distance between the endpoints of the walk $\pi$ on $G_k$.
Assume $|\gamma|=|\g(G_0)| <\oo$, and write
\begin{equation}\label{def01}
V_k(z)=\sum_{n=1}^{\infty} \frac1{n^{2z+1}} \langle\|\pi\|_k^2 \rangle_{n,k}, 
\end{equation}
where $\langle \cdot \rangle_{n,k}$ denotes the uniform average over the set 
$\Sigma_{n,k}$ of $n$-step SAWs of $G _k$ starting at mid-edges of $X_k$.
Similarly,
\begin{equation}\label{def02}
V_1^*(z)=\sum_{n=1}^{\infty} \frac1{n^{2z+1}}\frac{\sigma_{n,1}^*}{\sigma_{n,1}} 
\langle\|\pi\|_1^2 \rangle_{n,1}^*, 
\end{equation}
where $\s_\cdot^\cdot = |\Sigma_\cdot^\cdot|$, and 
$\langle\cdot \rangle_{n,1}^*$ averages over the subset $\Sigma_{n,1}^*$ of
$\Sigma_{n,1}$ containing $n$-step SAWs of $G_1$ that start in $X_0$
and end in $E_0$.
We  assume there exist constants $C_{i,k} \in (0,\oo)$ and a
slowly varying function $L$ such that
\begin{equation}\label{def2}
C_{1,k}L(n)n^{\gamma-1}\mu_k^n \leq \sigma_{n,k}
\leq C_{2,k}L(n)n^{\gamma-1}\mu_k^n, \qquad k=1,2.
\end{equation}
We shall in fact use slightly less than this.

Similarly to the proof of \eqref{gg1}, 
by \eqref{def02}--\eqref{def2}, there exists $C_1 <\oo$ such that
\begin{equation*}
V^*_1(z)\leq V_1(z)\leq C_1 V^*_1(z),
\end{equation*}
whence
\begin{equation}\label{gg2}
V_1(z)<\oo \quad \Leftrightarrow \quad V_1^*(z)<\oo.
\end{equation}

The contribution of $\pi\in \Sigma_{n,0}$ to $V_0(z)$ is
\begin{equation*}
\frac{1}{\sigma_{n,0}n^{2z+1}}\|\pi\|_0^2.
\end{equation*}
As explained previously, $\pi$ gives rise to
$2^n$ SAWs on $G _1$, making an aggregate contribution of
\begin{equation*}
\sum_{l=0}^n\binom n l 
\frac{1}{\sigma_{2n+l,1}(2n+l)^{2z+1}}(2\|\pi\|_0)^2
\end{equation*}
 to $V^*_1(z)$.
By \eqref{def2}, there exist constants $C_i >0$ such that
\begin{align*}
\frac{C_2}{n^{\gamma-1}L(n)}\sum_{l=0}^{n}\binom nl \left(\frac{1}{\mu_1}\right)^{2n+l}
&\leq\sum_{l=0}^{n}\binom nl \frac{1}{\s_{2n+l,1}}\\
&\leq \frac{C_3}{n^{\gamma-1}L(n)}\sum_{l=0}^{n}
\binom nl \left(\frac{1}{\mu_1}\right)^{2n+l}.
\end{align*} 

By Theorem \ref{thm:main2}(a),
\begin{equation*}
\sum_{l=0}^{n}\binom nl \left(\frac{1}{\mu_1}\right)^{2n+l}=\left(\frac{1}{\mu_{0}}\right)^n,
\end{equation*}
so that
\begin{equation*}
C_4(2^{-2z}\wedge 3^{-2z}) V_0(z)
\leq V_1^*(z)
\leq C_5 (2^{-2z}\vee 3^{-2z}) V_0(z).
\end{equation*}
Therefore, for $|z|<\oo$, 
$$
V_1^*(z)<\oo \quad \Leftrightarrow \quad V_0(z)<\oo.
$$
By \eqref{gg2}, $\nu(G_0) = \nu(G_1)$.

Any $\a\in\sA(G_0)$ acts in a natural way on $G_1=F(G_0)$.
For $k=0,1$ and $\a\in \sA$, 
let
\begin{align*}
Z_{\a, k}(x)=\sum_{\pi\in\Sigma_k(\a)}      x^{|\pi|},\quad
Z^*_{\a,k}(x)=\sum_{\pi\in\Sigma_k^*(\a)} x^{|\pi|},
\qquad x > 0,
\end{align*}
where $\Sigma_k(\a)$ (\resp, $\Sigma_k^*(\a)$) is the set of SAWs
of $G_k$ from mid-edges of $X_k$ (\resp,  $X_{k-1}$) to mid-edges of
$\a X_k$ (\resp,  $\a X_{k-1}$).  Assume $X_0$ and $\a X_0$ are disjoint.
As before,
\begin{equation}
Z_{\a,1}^*(x)=Z_{\a,0}(x^2(1+x))\label{rr1}
\end{equation}
and, as in \eqref{222},
\begin{equation}
Z_{\a,1}^*(x)
\leq Z_{\a, 1}(x)
\leq [1+2x+2x^2+2x^3]^2 Z_{\a,1}^*(x).\label{rr2}
\end{equation}
By \eqref{rr1} and Theorem \ref{thm:main2}(a), for  $w\in\RR$,
\begin{equation*}
\lim_{d_0(W_0,\a W_0)\rightarrow\infty} 
\Bigl[Z_{\a,0}(\mu_0^{-1})d_0(W_0,\a W_0)^{w}\Bigr]
=\infty
\end{equation*}
if and only if
\begin{equation*}
\lim_{d_1(W_1,\a W_1)\rightarrow\infty} 
\Bigl[Z_{\a, 1}(\mu_{1}^{-1})d_1(W_1, \a W_1)^w\Bigr]
=\infty,
\end{equation*}
where $d_k=d_{G_k}$. It follows that  $\eta(G_0) = \eta(G_1)$.
\end{proof}

\section{Proof of Theorem \ref{thm:main3}}\label{sec:proof3}

Let $G=(V,E)$ be a coloured bipartite graph satisfying the given assumptions.
The vertices of any SAW on  $G$ are alternately black and white.
The decorated graph $\wt G = (\wt V, \wt E)$ 
is obtained from $G$ by replacing each black vertex by a triangle,
as illustrated in Figure \ref{fig:fisher}.
The set $\wt V$ is coloured in the natural way: white vertices remain
white, and vertices of Fisher triangles are coloured black.

Let $W$ be a minimal fundamental domain of $G$, and let 
$X$ be the subset of $E$ comprising all edges incident 
to vertices in $W$. Write $\wt W=F(W)$, the set of vertices of the 
triangles formed by the Fisher transformations at black vertices in $W$,
and $\wt X$ for the set of edges of $\wt G$ incident to vertices in 
$\wt W$. It may be seen that
$\wt W$ is a fundamental domain for the coloured graph
$\wt G$. 
Recall that the mid-edges of $E$ may be viewed
as a subset of mid-edges of $\wt E$, and thus $E$ may be 
viewed as a subset of $\wt E$.

Let $s_n$ be the number of $n$-step SAWs of $\wt{G}$
starting at mid-edges in $\wt X$, and 
let $c_n$ be the number of $n$-step SAWs of $\wt{G}$ starting at a mid-edge of  $X$ and ending at a mid-edge of $E$. 
It is immediate that
\begin{equation}\label{rel1}
c_n\leq s_n.
\end{equation}
Any SAW counted in $s_n$ either lies within a single Fisher triangle,
or may be obtained by a 
$k$-step extension (with some $k\le 3$) at one or both endpoints 
of some SAW counted in one of  $c_n$, $c_{n-1}$, $c_{n-2}$, $c_{n-3}$.
Therefore,
\begin{equation}\label{rel2}
s_n\leq c_n + 4c_{n-1}+8c_{n-2}+12c_{n-3} + 18 |W|. 
\end{equation}
By \eqref{connconst},
 the limits $\lim_{n\to\infty}s_n^{1/n}$ and $\lim_{n\to\infty}c_n^{1/n}$ exist and, by \eqref{rel1}--\eqref{rel2}, 
these limits are equal.

A SAW is called \emph{even} if it has even length.
Let $\sE$ be the set of SAWs on $G$ starting at mid-edges of $X$,
and let $\sEe$ be the subset of $\sE$ comprising the even SAWs.
Let $x,y>0$. 
Each step of a SAW on $G$ is assigned weight $x$ at a black
vertex, and weight $y$ at a white vertex. 
Let 
\begin{equation*}
Z(x,y)=\sum_{\pi\in\sE }x^{|\pi_\bb|}y^{|\pi_\ww|},
\end{equation*}
where $|\pi_\bb|$ and $|\pi_\ww|$ are the numbers of 
black and white vertices visited by $\pi$. 
Similarly, let 
\begin{equation}\label{gg3}
\Ze(x,y)=\sum_{\pi \in \sEe}(xy)^{|\pi|/2}.
\end{equation}
It is clear by a decomposition of paths that
\begin{align*}
\Ze(x,y)&\leq Z(x,y),\\
Z(x,y)-\Ze(x,y)&\leq (2x+2y)(1+\Ze(x,y)).
\end{align*}
Hence, 
\begin{equation}\label{301}
\Ze(x,y)<\oo \quad \Leftrightarrow \quad Z(x,y)<\oo.
\end{equation}

We now introduce a third partition function $\wt Z$, namely of the set
$\wt\sE$ of SAWs on $\wt G$ starting at the mid-edges of $X$
and ending at mid-edges of $E$. Each step of such a SAW
traverses two half-edges, and is allocated a weight which depends
on these half-edges.  Let $p,q,r>0$. 
Whenever both half-edges belong to $\wt E\setminus E$,  the weight is $p$; 
if one half-edge is in $E$ and the other in $\wt E\setminus E $, 
the weight is $q$; if both half-edges are in $E$,  the weight is $r$. Then
\begin{equation*}
\wt{Z}(p,q,r) :=
\sum_{\pi\in \wt\sE} p^{|\pi_{p}|} q^{|\pi_q|}r^{|\pi_r|},
\end{equation*}
where $|\pi_p|$ is the number of $p$-steps, etc. By counting edges
of the different types,
\begin{equation*}
\wt{Z}(p,q,r)=Z(q^2(1+p),r).
\end{equation*}
By \eqref{301},
\begin{equation}\label{302}
\wt{Z}(p,q,r)<\oo \quad \Leftrightarrow \quad \Ze(q^2(1+p),r)<\oo.
\end{equation}
By \eqref{gg3},
\begin{equation*}
\Ze(q^2(1+p),r)
\begin{cases} <\oo &\text{if } q^2(1+p)r<\mu^{-2},\\
=\oo &\text{if } q^2(1+p)r> \mu^{-2},
\end{cases}
\end{equation*}
whence the radius of convergence of 
$\wt{Z}(x,x,x)=\sum_{n\ge 0} c_n x^n$ is the
root of the equation
\begin{equation*}
x^3(1+x)=\frac{1}{\mu^2}.
\end{equation*}
Theorem \ref{thm:main3}(a) follows.   Part (b) is proved is
a similar manner to the proof of Theorem \ref{thm:main2'}.

\section*{Acknowledgement} 
This work was supported in part by the Engineering
and Physical Sciences Research Council under grant EP/103372X/1.

\bibliography{fisher2}

\providecommand{\bysame}{\leavevmode\hbox to3em{\hrulefill}\thinspace}
\providecommand{\MR}{\relax\ifhmode\unskip\space\fi MR }
\providecommand{\MRhref}[2]{%
  \href{http://www.ams.org/mathscinet-getitem?mr=#1}{#2}
}
\providecommand{\href}[2]{#2}
\begin{thebibliography}{10}

\bibitem{bdgs}
R.~Bauerschmidt, H.~Duminil-Copin, J.~Goodman, and G.~Slade, \emph{Lectures on
  self-avoiding-walks}, Probability and Statistical Physics in Two and More
  Dimensions (D.~Ellwood, C.~M. Newman, V.~Sidoravicius, and W.~Werner, eds.),
  Clay Mathematics Institute Proceedings, vol.~15, CMI/AMS publication, 2012.

\bibitem{bdet}
C.~Boutillier and B.~de~Tili\`ere, \emph{The critical {Z}-invariant {I}sing
  model via dimers: locality property}, Commun. Math. Phys. \textbf{301}
  (2011), 473--516.

\bibitem{ds}
H.~Duminil-Copin and S.~Smirnov, \emph{The connective constant of the honeycomb
  lattice equals $\sqrt{2+\sqrt 2}$}, Ann. Math. \textbf{175} (2012),
  1653--1665.

\bibitem{fisher}
M.~E. Fisher, \emph{On the dimer solution of planar {I}sing models}, J. Math.
  Phys. \textbf{7} (1966), 1776--1781.

\bibitem{f}
P.~Flory, \emph{{Principles of Polymer Chemistry}}, Cornell University Press,
  1953.

\bibitem{g}
G.~R. Grimmett, \emph{Three theorems in discrete random geometry}, Probab.
  Surv. \textbf{8} (2011), 403--441.

\bibitem{GrL1}
G.~R. Grimmett and Z.~Li, \emph{Self-avoiding walks on regular graphs},
  (2012), preprint.

\bibitem{jmhII}
J.~M. Hammersley, \emph{Percolation processes {II. T}he connective constant},
  Proc. Camb. Phil. Soc. \textbf{53} (1957), 642--645.

\bibitem{jg}
I.~Jensen and A.~J. Guttman, \emph{Self-avoiding walks, neighbour-avoiding
  walks and trails on semiregular lattices}, J. Phys. A: Math. Gen. \textbf{31}
  (1998), 8137--8145.

\bibitem{zli}
Z.~Li, \emph{Local statistics of realizable vertex models}, Commun. Math. Phys.
  \textbf{304} (2011), 723--763.

\bibitem{li}
\bysame, \emph{Critical temperature of periodic {I}sing models}, Commun. Math.
  Phys. (2012), \url{http://arxiv.org/abs/1008.3934}.

\bibitem{ms}
N.~Madras and G.~Slade, \emph{{Self-Avoiding Walks}}, Birkh\"auser, Boston,
  1993.

\end{thebibliography}
\bibliographystyle{amsplain}

\end{document}